\documentclass[10pt]{amsart}

\usepackage{latexsym}
\usepackage{amsmath}
\usepackage{amsfonts}
\usepackage{amssymb}
\usepackage{mathrsfs}
\usepackage{bm}
\usepackage{comment} 
\usepackage{accents}
\usepackage[mathscr]{euscript}

\usepackage[usenames,dvipsnames,svgnames,table]{xcolor}

\usepackage{tikz}
\usepackage{tikz-cd}
\usetikzlibrary{matrix,arrows}
\usepackage{scalefnt}

\newcommand{\co}{\mathcal{O}}
\newcommand{\cl}{\mathcal{L}}
\newcommand{\cn}{\mathcal{N}}
\newcommand{\cc}{\mathcal{C}}

\newcommand{\ind}[2]{\mathcal{I}^{#2}_{#1}}
\newcommand{\res}[2]{\mathcal{R}^{#2}_{#1}}
\newcommand{\dres}[2]{\tilde{\mathcal{R}}^{#2}_{#1}}
\newcommand{\indPG}{\ind{P}{G}}
\newcommand{\resPG}{\res{P}{G}}
\newcommand{\dresPG}{\dres{P}{G}}

\newcommand{\dg}{\mathrm{DG}}

\newtheorem{theorem}{Theorem}[section]
\newtheorem{lemma}[theorem]{Lemma}

\newtheorem{proposition}[theorem]{Proposition}
\newtheorem{noname*}[theorem]{}

\theoremstyle{definition}
\newtheorem{definition}[theorem]{Definition}

\theoremstyle{remark}
\newtheorem{remark}[theorem]{Remark}

\numberwithin{equation}{section}

\DeclareMathOperator{\Hom}{Hom}
\DeclareMathOperator{\dgHom}{\mathcal{H}om}

\DeclareMathOperator{\Ext}{Ext}
\DeclareMathOperator{\uHom}{\underline{Hom}}

\DeclareMathOperator{\dm}{D^{\textup{b}}_{\textup{m}}}
\newcommand{\dgm}{\mathrm{D}^{\mathrm{b}}_{G, \mathrm{m}}}
\DeclareMathOperator{\db}{D^{\textup{b}}}

\newcommand{\dbeq}[2]{\mathrm{D}^{\mathrm{b}}_{#1}(#2)}

\newcommand{\dbg}{\mathrm{D}^{\mathrm{b}}_G}

\DeclareMathOperator{\vd}{\mathbb{D}}

\newcommand{\kb}{\mathrm{K}^\textup{b}}

\DeclareMathOperator{\ql}{\overline{\mathbb{Q}}_\ell}
\DeclareMathOperator{\uH}{\underline{H}}
\DeclareMathOperator{\cH}{H}
\DeclareMathOperator{\pure}{Pure}

\DeclareMathOperator{\ic}{IC}
\DeclareMathOperator{\pt}{pt}
\DeclareMathOperator{\perv}{Perv}

\DeclareMathOperator{\Lie}{Lie}

\DeclareMathOperator{\End}{End}
\newcommand{\R}{\Ql[W]\#\Sh}
\newcommand{\Rc}{\Ql [W(L)]\#\Sz}

\DeclareMathOperator{\Sh}{\mathsf{S}\mathfrak{t}^*}
\newcommand{\Sz}{\mathsf{S}\mathfrak{z}^*}
\DeclareMathOperator{\Ql}{\overline{\mathbb{Q}}_{\ell}}

\DeclareMathOperator{\dbc}{D_{c}^{b}}
\DeclareMathOperator{\constant}{\underline{\ql}}
\newcommand{\dual}[1]{\omega_{#1}}
\newcommand{\cusp}{\mathbf{c}}
\newcommand{\spec}{\mathrm{Spec}}

\begin{document}

\title[Derived Generalized Springer Correspondence]{Formality and Lusztig's Generalized Springer Correspondence}

\author{Laura Rider}
\address{Department of Mathematics\\
University of Georgia\\
Athens Georgia 30602}
\email{laurajoy@uga.edu}

\author{Amber Russell}
\address{Department of Mathematics and Actuarial Science\\
Butler University\\
Indianapolis Indiana 46208}
\email{acrusse3@butler.edu}


\begin{abstract} We prove a derived equivalence between each block of the derived category of sheaves on the nilpotent cone and the category of differential graded modules over a degeneration of Lusztig's graded Hecke algebra. Along the way, we construct and study a mixed version of the geometric category. This work can be viewed as giving a derived version of the generalized Springer correspondence.
\end{abstract}

\maketitle

\section{Introduction and Review of the generalized Springer correspondence}  

Let $G$ be a connected, reductive algebraic group defined over an algebraically closed field of good characteristic, and let $\cn$ denote its nilpotent cone. In this note, we complete the description of $\dbg(\cn)$, the $G$-equivariant derived category of sheaves on the nilpotent cone as initiated in \cite{Ri}. Lusztig's generalized Springer correspondence relates $\perv_G(\cn)$, the category of perverse sheaves on $\cn$, to relative Weyl group representations. We extend this to a derived equivalence between $\dbg(\cn)$ and the (perfect) derived category of  differential graded (dg) modules for the associated $\Ext$-algebra. This completes the proof of a version of the Soergel--Lunts conjecture (see \cite{Lu, So}) for the nilpotent cone. 

The category $\perv_G(\cn)$ is semisimple. However, two irreducible perverse sheaves $S$ and $S'$ may have ``geometric" extensions between them, i.e. $\Hom^i(S, S')$ need not equal 0 for $i>0$. In \cite{L1}, Lusztig classified the irreducible perverse sheaves on $\cn$ in terms of cuspidal data for $G$: each is a summand of a unique {\it Lusztig sheaf} $\mathbb{A}_{\cusp}$, the parabolic induction of some simple cuspidal perverse sheaf for a Levi subgroup of $G$. In fact, $\Hom^i(S, S')\neq 0$ if and only if $S$ and $S'$ are both summands of the same Lusztig sheaf. Hence, we have an orthogonal decomposition of the $G$-equivariant derived category according to Lusztig's classification of cuspidal data $\cusp$ up to $G$-conjugacy  \begin{equation}\label{eq:orthog}\hspace{.4in}\dbg(\cn) \cong \bigoplus_{\cusp/\sim} \dbg(\cn, \mathbb{A}_{\cusp}), \text{\hspace{.4in}\cite[Theorem 3.5]{RR}},\end{equation} where $\dbg(\cn, \mathbb{A}_{\cusp})$ denotes the triangulated subcategory of $\dbg(\cn)$ generated by the summands of $\mathbb{A}_{\cusp}$. Recently, several examples of related orthogonal decompositions have appeared. Achar, Henderson, Juteau, and Riche have made progress in the modular case \cite{AHJR}, Lusztig and Yun have proved an orthogonal decomposition for $\mathbb{Z}/m$-graded Lie algebras in \cite{LY}, and Gunningham has proved an orthogonal decomposition in the setting of equivariant D-modules on the Lie algebra in \cite{G}.  

Now we consider the Springer block. The perverse sheaves in the Springer correspondence occur as summands of what is known as the Springer sheaf, denoted $\mathbb{A}$. This is the Lusztig sheaf for the cuspidal datum when the Levi is a maximal torus. Let $W$ be the Weyl group of $G$, $T$ be a maximal torus in $G$, and $\mathfrak{t} = \mathrm{Lie}(T)$. Consider the graded algebra $\mathscr{A} = \Hom^\bullet(\mathbb{A}, \mathbb{A})$ isomorphic to a smash product algebra, $\R$ (see \cite{DR2, Kw} for defintion and details). In degree 0, this reduces to the ring isomorphism $\End(\mathbb{A})\cong\Ql[W]$ which yields the Springer correspondence linking the simple summands of $\mathbb{A}$ to irreducible $W$ representations, see \cite{BM1, L1}. 

Regard $\mathscr{A}$ as a dg-algebra with trivial differential. In \cite{Ri}, the first author proves that the Springer correspondence extends to an equivalence of derived categories \[\dbg(\cn, \mathbb{A}) \cong \dg(\mathscr{A}).\] There are many obstructions to constructing a triangulated functor $\dbg(\cn, \mathbb{A}) \rightarrow \dg(\mathscr{A})$ directly. For one thing, objects in $\dbg(\cn, \mathbb{A})$ are not complexes of sheaves. To overcome this difficulty, we construct and study a graded version of $\dbg(\cn, \mathbb{A})$ built from Deligne's category of mixed $\ell$-adic sheaves. This construction requires that Frobenius acts on $\Hom^i(\mathbb{A}, \mathbb{A})$ as multiplication by $q^{i/2}$.  This allows a proof of formality for an intermediate dg-algebra which is quasi-isomorphic to $\mathscr{A}$. Another important component of the proof is that the category from which the Springer sheaf is induced, $\dbeq{T}{\pt}$, is known to be formal by \cite{BL}. We use this category as a sort of dg-enhancement to $\dbg(\cn, \mathbb{A})$. This equivalence gives a description of one block in the orthogonal decomposition \eqref{eq:orthog} of $\dbg(\cn)$. 

Our present goal is to give a similar description for the other blocks in the decomposition. To that end, fix a cuspidal datum $\cusp = (L, \co_L, \cc)$ (see Defintion \ref{def:cuspdata}) and let $W(L) = \mathrm{N}_G(L)/L$. Let $Z(L)^\circ$ be the identity component of the center of $L$ and $\mathfrak{z}=\Lie(Z(L)^\circ)$.  Then $W(L)$ is a Weyl group (\cite[Theorem 5.9]{L0}) and $W(L)$ acts on $\mathfrak{z}$ by the reflection representation. Now, we consider the dg-algebra $\mathscr{A}_\cusp = \Hom^\bullet(\mathbb{A}_\cusp, \mathbb{A}_\cusp)$. Lusztig's work implies this algebra is isomorphic to a degeneration of the graded Hecke algebra from \cite{L3}: $\mathscr{A}_\cusp \cong \Rc$ (see \cite[Theorem 3.1]{K} for a proof of this isomorphism). In degree 0, this reduces to an isomorphism $\End(\mathbb{A}_\cusp)\cong \Ql[W(L)]$ proven in \cite{L1} which yields Lusztig's generalized Springer correspondence: a bijection between the simple summands of $\mathbb{A}_\cusp$ (up to isomorphism) and irreducible $W(L)$ representations. In the present note, we prove the derived equivalence (see Theorem \ref{thm:main}): \[\dbg(\cn, \mathbb{A}_\cusp)\cong \dg(\mathscr{A}_\cusp).\] We follow the procedure as in \cite{Ri} but must address the added complication of $\mathbb{A}_\cusp$ arising from a non-constant local system. In particular, we construct in Theorem \ref{lem:catinv} a graded version of $\dbg(\cn)$.

The problem of extending Lusztig's generalized Springer correspondence is also considered by Gunningham in \cite{G}. However, \cite{G} relies on tools from the setting of infinity categories. Furthermore, Gunningham focuses on the category of D-modules as opposed to constructible complexes. This approach allows deduction of an equivalence without considering Deligne's mixed sheaves and formality.

We expect that similar techniques can be applied in the setting of $K$-orbits on the flag variety $G/B$, since in this setting, general perverse sheaves are obtained by an induction-type procedure from some \textit{cuspidal} ones which are known to be clean.\footnote{This suggestion was communicated to the first author by W. Soergel.}

\subsection*{Organization} The paper is organized as follows. In Section \ref{sec:cuspblock}, we review the setup required to study Lusztig's generalized Springer correspondence. We also prove formality directly for a cuspidal block (see Proposition \ref{prop:cuspformal}). In Section \ref{sec:mix}, we study the Frobenius action on $\Hom^i(\mathbb{A}_\cusp, \mathbb{A}_\cusp)$ (see Propositions \ref{propcohovanish} and \ref{prop:frobin}). This is the key ingredient in the construction of the \emph{mixed}/\emph{graded} version of $\dbeq{G}{\cn, \mathbb{A}_\cusp}$ (see Theorem \ref{lem:catinv}). Then we prove a mixed version of the generalized Springer correspondence, Theorem \ref{thm:mixgen}, and finally our main result, Theorem \ref{thm:main}. 

\section{Preliminaries}\label{sec:cuspblock}
\subsection{Notation}\label{sec:prelim}
We fix an algebraically closed field $k$ of positive characteristic. All varieties we consider are defined over $k$ except in Section \ref{sec:mix} when we need to employ mixed sheaves as developed in \cite{De} and \cite[Section 5]{BBD}. For an action of an algebraic group $G$ on a variety $X$, we consider the categories, denoted $\perv_G(X)\subset \dbeq{G}{X}$, of $G$-equivariant perverse sheaves on $X$ and the $G$-equivariant (bounded) derived category of sheaves on $X$. See \cite{BL} for background and definitions related to these categories. All of our sheaves will have $\ql$ coefficients. For $\mathscr{F}, \mathscr{G}\in \dbeq{G}{X}$, we let $\Hom^i(\mathscr{F}, \mathscr{G}):=\Hom_{\dbeq{G}{X}}(\mathscr{F}, \mathscr{G}[i])$. All sheaf functors are understood to be derived. We denote the constant sheaf on $X$ by $\constant_X$ or just $\constant$ when there is no ambiguity. Let $j:\co\hookrightarrow X$ be an inclusion of a locally closed smooth subset and let $\cl$ be an irreducible local system on $\co$. Then the intermediate extension of $\cl$ to $X$, denoted $\ic(\co, \cl):=j_{!*}\cl[\dim\co]$, is a simple perverse sheaf on $X$. We denote by $\vd$ the verdier duality functor preserving the category of perverse sheaves. For a local system $\cl$, we let $\cl^*$ denote the local system given by $R\dgHom(\cl, \constant)$.

Our main theorem relates a geometric derived category to an algebraic category of differential graded modules. For a differential graded (or dg from now on) algebra $\mathscr{A}$, we will consider the derived category of finitely generated dg-modules denoted by $\dg(\mathscr{A})$. See \cite{BL} for definitions.

\subsection{Induction and Restriction Functors} 
Let $G$ be a connected, reductive algebraic group defined over $k$ of good characteristic. We consider $\cn$, its nilpotent cone with the adjoint $G$-action. Let $P$ be a parabolic subgroup of $G$ with Levi decomposition $P = LU_P$. We denote by $\cn_L$ the nilpotent cone for $L$ and $\mathfrak{u}_P = \Lie(U_P)$. We consider the following $G$-varieties and $G$-maps \[\widetilde{\cn}^P:=G\times^P(\mathfrak{u}_P + \cn_L) \hspace{.2cm}\textup{ and }\hspace{.2cm} \mathcal{C}_P:=G\times^P\cn_L,\]\[\cn\xleftarrow{\mu}\widetilde{\cn}^P\stackrel{\pi}{\longrightarrow}\mathcal{C}_P.\] The variety $\widetilde{\cn}^P$ is called a partial resolution of $\cn$ and is studied in \cite{BM2}. Note that $\mu$ is proper and $\pi$ is smooth of relative dimension $d_P = \dim U_P$, so we have $\mu_! \cong \mu_*$ and $\pi^! = \pi^*[2d_P](d_P)$. We consider the following functors \begin{equation}\label{eq:functors} \begin{split} \indPG = \mu_!\pi^*[d_P](\tfrac{d_P}{2}) \cong \mu_*\pi^![-d_P](\tfrac{-d_P}{2}), \\ 
\resPG = \pi_*\mu^![-d_P](\tfrac{-d_P}{2}), \hspace{.2cm}\dresPG=\pi_!\mu^*[d_P](\tfrac{d_P}{2}), \end{split}\end{equation} which we will refer to as induction and restriction functors. We have adjoint pairs $(\indPG, \resPG)$ and $(\dresPG, \indPG)$. Lusztig proves that the functors $\indPG$, $\resPG$, and $\dresPG$ are exact with respect to the perverse $t$-structure. See \cite[Theorem 4.4]{L2} or for more general coefficient rings (Noetherian commutative ring of finite global dimension), see \cite[Proposition 4.7]{AHR}. The notation $(i)$ for a half-integer $i$ indicates a \emph{Tate twist}; see Section \ref{sec:mix}. We include the Tate twist in the definition of these functors in anticipation of when we apply them in the mixed setting in Section \ref{sec:mix}. With the above twists, the functors $\indPG, \resPG$, and $\dresPG$ preserve weights.

\begin{remark}\label{rm:ind} Let $H$ be a closed subgroup of $G$, with $d_H = \dim H$ and $d_G = \dim G$. Let $X$ be an $H$-space, and $\nu: X\hookrightarrow G\times^H X$. Then equivariant induction equivalence is given by $\nu^*\circ \mathrm{For}^G_H: \dbeq{G}{G\times^H X} \rightarrow \dbeq{H}{X}$, \cite[2.6.3]{BL}. This functor is exact with respect to the constructible t-structure, and its shift by $d_H-d_G$ preserves the perverse t-structure (\cite[5.1]{BL}). Let $\mathbb{V}$ denote the inverse. Obviously, $\nu^!\circ \mathrm{For}^G_H: \dbeq{G}{G\times^H X} \rightarrow \dbeq{H}{X}$ is also an equivalence, though now it does not preserve the constructible t-structure. The two are related by $\nu^*\cong \nu^![2(d_G - d_H)]$. Hence, the inverse to $\nu^!$ is isomorphic to $\mathbb{V}[2(d_H - d_G)]$ (or $\mathbb{V}[2(d_H- d_G)](d_H - d_G)$ in the mixed setting).\end{remark} Since $U_P$ acts trivially on $\cn_L$ and is contractible, Remark \ref{rm:ind} gives an equivalence of categories $\dbeq{G}{\mathcal{C}_P}\cong\dbeq{L}{\cn_L}$. Thus, often we will think of the induction (respectively, restriction) functor as having domain (respectively, codomain) $\dbeq{L}{\cn_L}$:
\[\indPG: \dbeq{L}{\cn_L}\rightarrow \dbeq{G}{\cn} \textup{ and }\resPG, \dresPG:  \dbeq{G}{\cn}\rightarrow \dbeq{L}{\cn_L}. \]

\subsection{Formality for a Cuspidal Block}
\begin{definition}\label{def:cuspdata} A simple perverse sheaf $\ic(\co,\cc)\in\dbeq{G}{\cn}$ is called \textit{$G$-cuspidal} if $\resPG(\ic(\co,\cc))= 0$ for all proper parabolics $P$ in $G$. A \textit{cuspidal datum} for $G$ is a tuple $\cusp = (L, \co_L, \cc)$ where $L$ is a Levi subgroup of $G$, $\co_L$ is an $L$-orbit in $\cn_L$ so that $\ic(\co_L, \cc)$ is $L$-cuspidal. We say two cuspidal data $\cusp$ and $\cusp '$ are equivalent if they are conjugate in $G$, and in this case, we write $\cusp \sim \cusp '$.\end{definition} 

Associated to each cuspidal datum $\cusp=(L, \co_L, \cc)$ is the semisimple perverse sheaf $\mathbb{A}_{\cusp} = \indPG(\ic(\co_L, \cc))[d_P](\tfrac{d_P}{2})$, which we refer to as a \textit{Lusztig sheaf}. We will often work with the triangulated subcategory of $\dbg(\cn)$ generated by the simple summands of $\mathbb{A}_\cusp$ which we denote by $\dbeq{G}{\cn, \mathbb{A}_{\cusp}}
$. We call a block $\dbg(\cn, \mathbb{A}_\cusp)$ in the orthogonal decomposition of \eqref{eq:orthog} \textit{cuspidal} if $\cusp = (G, \co_G, \cc)$ and so, $\mathbb{A}_\cusp = \ic(\co_G, \cc)$ is a simple perverse sheaf, which we denote by $\ic_\cusp$. Let $j:\co_G\hookrightarrow \cn$ be the inclusion. A cuspidal simple perverse sheaf is \textit{clean}, which means that $\ic_\cusp\cong j_!\cc[\dim\co_G] \cong j_*\cc[\dim\co_G]$. (Cleanness for cuspidals is due to Lusztig. See \cite[Proposition 4.2]{RR} for a proof.)

The following lemma can be deduced from Lusztig \cite{L3}. A proof of the statement (along with a study of the Frobenius action) can be found in \cite[Lemma 4.4]{RR}. However, the $\Ext$ computation given there assumes that the cuspidal local system is rank one, which is not always true. Cuspidal local systems are rank one in every simple type \emph{except} type D, where their rank can be a power of 2. 

\begin{lemma}\label{lem:ext} Let  $\cusp = (G, \co_G, \cc)$ and so, $\mathbb{A}_\cusp = \ic(\co_G, \cc)$. We have an isomorphism  $\mathscr{A}_\cusp \cong \Sz$, where $\mathfrak{z} = \Lie(Z(G)^\circ).$
\end{lemma}

\begin{proof} We begin as in \cite[Lemma 4.4]{RR} to get $\Hom^\bullet (\ic_\cusp, \ic_\cusp) \cong \cH^\bullet_{G}(\co_G, \cc\otimes\cc^\vee)$. Let $\pt\in\co_G$. Then $\co_G = G\cdot \pt \cong G\times_S \pt$ where $S$ denotes the centralizer in $G$ of $\pt$, and $S^\circ$ denotes its identity component. Then equivariant induction gives $ \cH^\bullet_{G}(\co_G, \cc\otimes\cc^\vee) \cong \cH^\bullet_{S}(\pt, (\cc\otimes\cc^\vee)_{\pt})$.  The stalk $\mathbb{E}: =\cc_{\pt}$ is naturally an irreducible representation of $S/S^\circ$. Recall $Z(G)^\circ$ denotes the identity component of the center of $G$. Now we apply an argument similar to \cite[1.12 (b)]{L3}: 
\begin{eqnarray*}
\cH^\bullet_{S}(\pt, \mathbb{E}\otimes\mathbb{E}^*) \cong (\cH^\bullet_{S^\circ}(\pt, \mathbb{E}\otimes\mathbb{E}^*))^{S/S^\circ}, & \text{by \cite[1.9]{L3}},\\
										\cong(\cH^\bullet_{Z(G)^\circ}(\pt, \mathbb{E}\otimes\mathbb{E}^*))^{S/S^\circ}, & \text{$S^\circ/Z(G)^\circ$ is unipotent, \cite[Prop. 2.8]{L1}},\\
										 \cong (\cH^\bullet_{Z(G)^\circ}(\pt)\otimes\mathbb{E}\otimes\mathbb{E}^*)^{S/S^\circ}, & \text{$Z(G)^\circ$ acts trivially on $\mathbb{E}$}\\
										  \cong \cH^\bullet_{Z(G)^\circ}(\pt)\otimes(\mathbb{E}\otimes\mathbb{E}^*)^{S/S^\circ}, & \text{$S/S^\circ$ acts trivially on $\mathfrak{z}$,}\\
										  \cong \cH^\bullet_{Z(G)^\circ}(\pt), & \text{$(\mathbb{E}\otimes\mathbb{E}^*)^{S/S^\circ}\cong\ql$}.
										\end{eqnarray*}

\end{proof}

The following proposition establishes formality for a cuspidal block $\dbg(\cn, \ic_\cusp)$. Understanding the cuspidal case will be instrumental in proving formality in the other cases. 

\begin{proposition} \label{prop:cuspformal}There is an equivalence $\dbg(\cn, \ic_\cusp)\cong \dbeq{Z}{\pt}$, where $Z = Z(G)^\circ$. In particular, a cuspidal block is formal.
\end{proposition}
\begin{proof}
Let $\pt\in\co_G$. Then $\co_G = G\cdot \pt \cong G\times_S \pt$ where $S$ denotes the centralizer in $G$ of $\pt$, and $S^\circ$ denotes its identity component. First, since the nilpotent orbit $\co_G$ supports a cuspidal local system, the group $S^\circ/Z$ is unipotent by \cite[Proposition 2.8]{L1}. This implies we have an equivalence $\dbeq{S^\circ}{\pt} \cong \dbeq{Z}{\pt}$. Our proof proceeds by the following sequence of equivalences:
 
\[\dbeq{S^\circ}{\pt}\stackrel{\mathrm{For}}{\longleftarrow} \dbeq{S}{\pt, \mathrm{triv}}\stackrel{A}{\rightarrow} \dbeq{G}{\co_G, \mathrm{triv}} \stackrel{-\otimes\cc}{\longrightarrow}\dbeq{G}{\co_G, \cc}  \stackrel{j_!}{\rightarrow}\dbeq{G}{\cn, \ic_\cusp}.\]
Arrow $A$ is the perverse shift of equivariant induction equivalence restricted to the triangulated subcategory generated by the trivial representation of $S/S^\circ$. To see that the above functors indeed give equivalences, we simply note that they all take the (only) simple object to the next (hence, each equivalence is exact with respect to the perverse t-structure), and that $\End^\bullet$ of these simple objects coincides in all categories. This is outlined in the proof of computing the Frobenius action in \cite[Lemma 4.4]{RR} and Lemma \ref{lem:ext}. Exactness of $j_!$ restricted to this subcategory is implied by cleanness of $\cc$. Formality for torus-equivariant sheaves on a point follows from \cite[12.4.6]{BL}. 
\end{proof}

\section{Frobenius Actions}\label{sec:mix}
In what follows, we must employ mixed sheaves as developed in \cite[Section 5]{BBD}. Our schemes/varieties are defined over the fields $\mathbb{F}_q$ and $\bar{\mathbb{F}}_q$. We follow the convention of \cite{BBD}: the index $\circ$ denotes an object defined over $\mathbb{F}_q$, while removal of the index $\circ$ indicates the extension of scalars to $\bar{\mathbb{F}}_q$. If $X_\circ$ is a scheme over $\mathbb{F}_q$, then $X : = X_\circ \times_{\spec \mathbb{F}_q}\spec \bar{\mathbb{F}}_q$, and there is an extension of scalars functor $\eta: \dbc(X_\circ)\rightarrow\dbc(X)$. On $X_\circ$, we will only consider sheaves in the triangulated subcategory of \textit{mixed complexes} $\dm(X_\circ)\subset \dbc(X_\circ)$. See \cite[5.1.5]{BBD} for a definition and properties. When a group $G_\circ$ acts on $X_\circ$, we will consider the mixed $G_\circ$-equivariant derived category, but will suppress the $\circ$ index on the $G$ in our notation: $\dgm(X_\circ)$.

For $\mathscr{F}_\circ\in\dm(X_\circ)$, the extension to $\bar{\mathbb{F}}_q$ is endowed with an isomorphism $\mathscr{F}\stackrel{\sim}{\rightarrow}\mathrm{Fr}^*\mathscr{F}$ where $\mathrm{Fr}: X\rightarrow X$ is the Frobenius map. We denote by $\uHom^i(\mathscr{F}, \mathscr{F}')$ the vector space $\Hom^i_{\dbc(X)}(\mathscr{F}, \mathscr{F}')$ endowed with the action of Frobenius induced from the $\mathbb{F}_q$ structure of $\mathscr{F}_\circ$ and $\mathscr{F}'_\circ$. We also need notation to denote Tate twist: \label{tate}$\mathscr{F}_\circ(i)$ denotes tensor product of $\mathscr{F}_\circ$ with the $i$th power of the Tate sheaf. 

Let $G_\circ$ be a connected, reductive algebraic group split over $\mathbb{F}_q$. We also consider $\mathbb{F}_q$-versions of other varieties appearing in previous sections with the obvious notation. We assume that $\mathbb{F}_q$ is large enough (by taking a finite field extension if necessary) so that all nilpotent $G_\circ$ orbits have geometric points. 

\subsection{A choice of mixed cuspidal local systems}
Suppose $\ic(\co, \cl)$ is $G$-cuspidal. Fix a closed, Frobenius fixed point $x\in\co$. By \cite[Proposition 12.1.2]{Sp}, the group $Z_\circ = (Z_{G}(x))_\circ$, an $\mathbb{F}_q$-version of the centralizer of $x$ in $G$, is defined and $\co_\circ =G_\circ\cdot x = G_\circ \times_{Z_\circ} \{x\}$. Similarly, \cite[Proposition 12.1.1]{Sp} implies the identity component $Z_\circ^o$ is defined, so we let $\tilde{\co}_\circ := G_\circ \times_{Z^o} \{x\}$. Then there is an obvious morphism $\pi: \tilde{\co}_\circ \rightarrow \co_\circ$ which is a finite, \'etale covering, with covering group $Z_\circ/Z_\circ^o \cong \pi_{1}^{G, \textup{\'et}}(\co)$. In particular, $\pi$ is proper and small. We define the local system $\mathcal{R}_\circ:=\pi_!\constant_{\tilde{\co_\circ}}$. It is pure of weight 0 since it is the proper push-foward of $\constant_{\tilde{\co_\circ}}$, \cite[5.1.14]{BBD}. Although $\mathcal{R}_\circ$ need not be semisimple, the local system $\eta( \mathcal{R}_\circ) = \mathcal{R}$ is semisimple by \cite[Theorem 5.3.8]{BBD} and it corresponds to the regular representation of $\pi_{1}^{G, \textup{\'et}}(\co)$. Moreover, we have a short exact sequence: \[0\rightarrow \uHom(S, S')_{\mathrm{Fr}}\rightarrow \Hom_{\dgm(\co_\circ)}^1(S, S')\rightarrow \uHom^1(S, S')^\mathrm{Fr}\rightarrow 0\] for any two $S, S'\in \dgm(\co_\circ).$ For simple perverse sheaves $S$ and $S'$ with $\eta S\not\cong \eta S'$ in $\dbg(\co)$, the first term vanishes since it is a quotient of the vector space $\Hom_{\dbg(\co)}(\eta S, \eta S')$ which is trivial by Schur's lemma. Similarly, the last term vanishes since the $G$-equivariant cohomology of a nilpotent orbit is concentrated in even degrees. Hence, for each irreducible representation $\chi$ of  $\pi_{1}^{G, \textup{\'et}}(\co)$, there is a summand $\cl_\circ^\chi\otimes E_n$ with $\eta(\cl^\chi_\circ)\cong \cl^\chi$ and $E_n$ a $\ql$ vector space of dimension $\mathrm{rank}(\cl^\chi)$ with possibly non-trivial, but unipotent Galois action, so there is an injection $\cl^\chi_\circ\hookrightarrow \cl_\circ^\chi\otimes E_n \hookrightarrow \mathcal{R}_\circ$, where Frobenius acts trivially on $\cl^\chi_\circ$. Similarly, for any cuspidal datum $(L, \co_L, \cc)$, there exists a choice of local system $\cc_\circ$ which has trivial Frobenius action. We fix this choice once and for all for each cuspidal datum. This allows us to define specific $\mathbb{F}_q$-versions of $\ic_\cusp$ and $\mathbb{A}_\cusp$ for a general cuspidal datum $\cusp$. Moreover, we note that $\mathbb{A}_{\cusp \circ}$ is pure of weight 0 since $\ind{P}{G}$ preserves weights. 

\begin{lemma} \label{lem:fqclean}Any $\mathbb{F}_q$-version $\cc_\circ$ of a cuspidal local system $\cc$ is clean. 
\end{lemma}
\begin{proof} Let $\ic_\circ$ be the IC-extension of $\cc_\circ$, and let $S_\circ$ be any simple perverse sheaf in $\dbeq{G, m}{\cn_\circ}$ so that $\eta(S_\circ)\not\cong\eta(\ic_\circ).$ Denote by $\mathcal{T}$ the triangulated subcategory of $\dbeq{G, m}{\cn_\circ}$ generated by all simple perverse sheaves $S'_\circ$ with $\eta(S'_\circ)\cong\eta(\ic_\circ)$.  We have the following short exact sequences

\[0\rightarrow \uHom^{i-1}(S, S')_{\mathrm{Fr}}\rightarrow \Hom_{\dbeq{G, m}{\cn_\circ}}^{i}(S_\circ, S'_\circ)\rightarrow \uHom^{i}(S, S')^\mathrm{Fr}\rightarrow 0\] 
\[0\rightarrow \uHom^{i-1}(S', S)_{\mathrm{Fr}}\rightarrow \Hom_{\dbeq{G, m}{\cn_\circ}}^{i}(S'_\circ, S)\rightarrow \uHom^{i}(S', S)^\mathrm{Fr}\rightarrow 0.\] 

The first and third term vanish in each short exact sequence by \cite[Theorem 3.5]{RR}, so the middle term vanishes as well. Hence, $\mathcal{T}$ is orthogonal to the `rest' of $\dbeq{G, m}{\cn_\circ}$. In particular, $\mathcal{T}$ contains no simple perverse sheaves with support strictly smaller than $\ic_\circ$. Hence, the proof of \cite[Proposition 4.2]{RR} applies without change.
\end{proof} 

\begin{remark} The technique in the proof of Lemma \ref{lem:fqclean} implies an orthogonal decomposition of $\dbeq{G, m}{\cn_\circ}$ similar to \cite[Theorem 3.5]{RR} where each piece in the decomposition corresponds to a single cuspidal datum. However, in the mixed setting, it is likely that each of these pieces can be refined further according to the Galois action.\end{remark}

\subsection{Semisimplicity of Frobenius} Fix a cuspidal datum $\cusp = (L, \co_L, \cc)$ with Lusztig sheaf $\mathbb{A}_\cusp$ and parabolic subgroup $P$. To study the Frobenius action on $\uHom(\mathbb{A}_{\cusp, \circ}, \mathbb{A}_{\cusp, \circ})$, we use its relationship with the cohomology of a generalized Steinberg variety, $Z_\circ = \widetilde{\cn}_\circ^P\times_{\cn_\circ}\widetilde{\cn}_\circ^{P}$. For the rest of this section, we let $S_\circ = \mathbb{V} \ic_{\cusp\circ}$, and $\bar{S}_\circ = \mathbb{V}\vd\ic_{\cusp\circ}$. Consider the composition \begin{equation}\label{map:tau}\tau : \widetilde{\cn}_\circ^P\times_{\cn_\circ}\widetilde{\cn}_\circ^{P}\stackrel{i}{\hookrightarrow}\widetilde{\cn}_\circ^P\times\widetilde{\cn}_\circ^{P}\stackrel{\pi\times\pi}{\longrightarrow}\cc_{P,\circ}\times\cc_{P,\circ}.\end{equation} In the following proposition, we compute the dimension of  $\cH_G^0(Z_\circ, \tau^! \bar{S}_\circ\boxtimes S_\circ)$. Consider the partial flag variety $G_\circ/P_\circ$. Let $W_L$ be the Weyl group of the Levi $L$. Then the $G_\circ$-orbits in $G_\circ/P_\circ\times G_\circ/P_\circ$ are in bijection with coset representatives $W_L\backslash W/ W_L$. Since $G_\circ$ is split, this is defined over $\mathbb{F}_q$. Consider the following diagram:  \begin{center}
\begin{tikzcd}
Z_w:=\varpi^{-1}(\mathcal{O}_w)\arrow{d}{\varpi}\arrow[hookrightarrow]{r}{\iota_w} &  Z\arrow{d}{\varpi}\\
 \mathcal{O}_w\arrow[hookrightarrow]{r}{} & G/P\times G/P\arrow{d}{}\\
  & \pt
\end{tikzcd}
\end{center} According to \cite{DR1}, the irreducible components of $Z_\circ$ are exactly the closures of $Z_{w,\circ}$, and so are in bijection with $G_\circ$ orbits of $G_\circ/P_\circ\times G_\circ/P_\circ$, and hence with coset representatives $W_L\backslash W/ W_L$. We call a $G_\circ$-orbit $\co_{w,\circ}$ of $G_\circ/P_\circ\times G_\circ/P_\circ$ \textit{good} if there is $(P', P'')\in\co_{w,\circ}$ such that $P'$ and $P''$ share a Levi subgroup. Similarly, we say that a double coset $w\in W_L\backslash W/W_L$ is \textit{good} or an irreducible component $Z_{w,\circ}$ is \textit{good} if $w$ corresponds to a good $G_\circ$-orbit of $G_\circ/P_\circ\times G_\circ/P_\circ$. Lusztig proves in \cite[Section 5]{L0} that the set of good cosets in $W_L\backslash W/ W_L$ identifies with the Coxeter group $W(L)$. Note that Lusztig uses the term ``distinguished'' in \cite{L0} instead of ``good''. 
 
The proof of the following Proposition is strongly influenced by \cite[Proposition 4.7]{L3}, but Lusztig's proof doesn't apply directly in our situation. In order to get that the vector space (in his case, equivariant homology) decomposes according to the irreducible components, Lusztig uses that the algebra has cohomology only in even degrees. Of course, this is true for $\uHom^\bullet(\mathbb{A}_\cusp, \mathbb{A}_\cusp)$, but in the mixed setting $\Hom_ {\dbeq{G, m}{\cn}}^1(\mathbb{A}_{\cusp\circ}, \mathbb{A}_{\cusp\circ})\neq0.$ 

\begin{proposition}\label{propcohovanish} Let $Z_\circ$ be as above.  Then, $\cH^0(Z_{w\circ}, \iota_w^!\tau^! \bar{S}\boxtimes S) = 0$ when $w$ is bad and is isomorphic to $\Hom(\ic_{\cusp\circ}, \ic_{\cusp\circ})$ when $w$ is good. In particular, $\dim\cH^0(Z_\circ, \tau^! \bar{S}\boxtimes S) = |W(L)|$.\end{proposition}

\begin{proof} Throughout, we work in the mixed derived category. For simplicity of notation, we drop the subscript $\circ$. 

First we compute the (shriek) restrictions to the irreducible components. We apply an approach similar to \cite[Proposition 3.2]{RR}, and recall the setup. Let $n\in G$ so that (1)  $L$ and $nLn^{-1}$ share a maximal torus and (2) the point $(P, nP)$ is in the $G$-orbit of $G/P\times G/P$ corresponding to $w$. Let $Z(n):=\varpi^{-1}((P, nP))\cong \cn_P\cap \cn_{P'}$.

Recall the following commutative diagram from \cite[Proposition 3.2]{RR} where $L' = nLn^{-1}, P'= nPn^{-1}, f: \cc_{P'} \stackrel{\sim}{\rightarrow}\cc_{P},$ and $E=\cn_{P'\cap L}\times_{\cn_{L\cap L'}}\cn_{P\cap L'} $.
\begin{center}
\begin{tikzpicture}[description/.style={fill=white,inner sep=1.5pt}] 
\matrix (m) [matrix of math nodes, row sep=2.5em,
column sep=2em, text height=1ex, text depth=0.25ex, nodes in empty cells]
{  Z &\cc_P\times \cc_{P} &\cc_P\times \cc_{P'}  \\ 
      &                                    & \cn_L\times\cn_{L'}\\
      Z(n) &\cn_P\cap\cn_{P'}& E\\};
\path[->,font=\scriptsize]
(m-1-1) edge node[above]{$\tau$}(m-1-2)
 (m-1-2) edge node[above]{$\sim$}(m-1-3)
(m-3-1) edge node[above]{$\sim$} (m-3-2)
(m-3-2) edge node[above]{$\alpha$} (m-3-3);

\path[right hook->,font=\scriptsize]
(m-3-1) edge node[left]{$\iota_n$} (m-1-1)
(m-2-3) edge node[right]{$\nu_1\times\nu_2$}(m-1-3)
(m-3-3) edge node[right]{$\Delta$}(m-2-3);
\end{tikzpicture}\end{center}

Let $S' = f^*S$. The map $\alpha$ is smooth of relative dimension $d_{P\cap P'}= \dim U_{P\cap P'},$ hence $\alpha^*[2d_{P\cap P'}](d_{P\cap P'})\cong\alpha^!$. Now we apply a special case of Braden's hyperbolic localization: equation (1) in \cite[Section 3]{Br}. In the diagram below, we let the multiplicative group $\mathbb{G}_m$ act on all varieties with compatible positive weights. Let $e: \{(0,0)\}\hookrightarrow E$ be the inclusion of the fixed point of this action, and let $\ell: E\rightarrow\{(0,0)\}$ be the map that sends every point to its limit.  \begin{center}
\begin{tikzpicture}[description/.style={fill=white,inner sep=1.5pt}] 
\matrix (m) [matrix of math nodes, row sep=2.5em,
column sep=2em, text height=1ex, text depth=0.25ex, nodes in empty cells]
{   \{0\} & \{(0, 0)\}   \\ 
          \cn_P\cap\cn_{P'} & E\\};
\path[->,font=\scriptsize, >=angle 90]
(m-1-1) edge[bend right=30]  node[left]{$b$} (m-2-1)
		edge node[above]{=}(m-1-2)
(m-2-1) edge node[above]{$\alpha$}(m-2-2)
(m-2-2) edge[bend right=30] node[right]{$\ell$}(m-1-2)
(m-2-1) edge[bend right = 30] node[right]{$a$}(m-1-1);

\path[right hook->,font=\scriptsize, >=angle 90]
(m-1-2) edge[bend right=30] node[left]{$e$}(m-2-2);
\end{tikzpicture}
\end{center} Then hyperbolic localization implies that we have isomorphisms $e^!\cong \ell_!$ and $a_! \cong b^!$. Furthermore, the diagram commutes. Combining these, we see that $a_!\alpha^!\cong b^!\alpha^! \cong e^! \cong \ell_!.$  By applying Verdier duality and since $\alpha$ is smooth, we also have \begin{equation}\label{eq:hyploc} a_*\alpha^!\cong \ell_*[2d_{P\cap P'}](d_{P\cap P'}).\end{equation}

Recall from \cite[Theorem 3.1]{DR1} that $Z_w \cong G\times^{P\cap P'} \cn_P\cap\cn_{P'}$.  Consider the map $\nu_3:\cn_P\cap\cn_{P'}\hookrightarrow Z_w$. Equivariant induction gives \[\cH^i_G(Z_w, \iota_w^!\tau^!(\bar{S}\boxtimes S))\cong \Hom^i_{\dbeq{P\cap P'}{Z(n)}}(\nu_3^!\constant, \nu_3^!\iota_w^!\tau^!(\bar{S}\boxtimes S)).\] Furthermore by Remark \ref{rm:ind}, we have that $\nu_3^!\constant \cong\constant[-2d_1](-d_1)$, where $d_1 = \dim G -\dim P\cap P'$. Since the diagram commutes, we have that \begin{equation}\label{eq:indtofib}\cH^i_G(Z_w, \iota_w^!\tau^!(\bar{S}\boxtimes S))\cong \cH^i_{P\cap P'}(Z(n), \alpha^!\Delta^!(\nu_1\times\nu_2)^!(\bar{S}\boxtimes S')[2d_1](d_1)).\end{equation}

Step 1: If $w$ is bad, $\cH_G^0(Z_{w\circ}, \iota_w^!\tau^! \bar{S}\boxtimes S)=0$.

Assume that $w$ is bad. Then, the parabolics $P$ and $P' = nPn^{-1}$ do not share a common Levi. Hence, $P'\cap L\times P\cap L'$ is a proper parabolic subgroup of $L\times L'$. In this case, an argument Verdier dual to \cite[Proposition 3.2]{RR} implies that $\ell_* \Delta^!(\nu\times\nu')^!\bar{S}\boxtimes S' = 0$. Hence, hyperbolic localization implies that  $a_*\alpha^! \Delta^!(\nu\times\nu')^! \bar{S}\boxtimes S' = 0.$ Finally, \eqref{eq:indtofib} yields Step 1.

Step 2: If $w$ is good, $\cH_G^0(Z_{w\circ}, \iota_w^!\tau^! \bar{S}\boxtimes S)\cong \Hom(\ic_\cusp, \ic_\cusp)$. Now we assume that $w$ is good. Then, the parabolics $P$ and $P' = nPn^{-1}$ share the Levi $L$. In this case, we see that $E = \cn_L$ and the above diagram becomes
\begin{center}
\begin{tikzpicture}[description/.style={fill=white,inner sep=1.5pt}] 
\matrix (m) [matrix of math nodes, row sep=2.5em,
column sep=2em, text height=1ex, text depth=0.25ex, nodes in empty cells]
{  Z &\cc_P\times \cc_{P} &\cc_P\times \cc_{P'}  \\ 
      &                                    & \cn_L\times\cn_{L}\\
      Z(n) &\cn_P\cap\cn_{P'}&\cn_L = E\\};
\path[->,font=\scriptsize]
(m-1-1) edge node[above]{$\tau$}(m-1-2)
 (m-1-2) edge node[above]{$\sim$}(m-1-3)
(m-3-1) edge node[above]{$\sim$} (m-3-2)
(m-3-2) edge node[above]{$\alpha$} (m-3-3);

\path[right hook->,font=\scriptsize]
(m-3-1) edge node[left]{$\iota_n$} (m-1-1)
(m-2-3) edge node[right]{$\nu_1\times\nu_2$} (m-1-3)
(m-3-3) edge node[right]{$\Delta$} (m-2-3);
\end{tikzpicture}\end{center} Recall that $a: Z(n)\rightarrow \pt$ and $\ell: \cn_L\rightarrow \pt$. We also have that $\nu_1^!(\bar{S})\cong (\vd \ic_\cusp)[-2d_P](-d_P)$ and $\nu_2^!(S')\cong \ic_\cusp[-2d_P](-d_P)$ by Remark \ref{rm:ind}. Furthermore, $-4d_P+2d_1 = -2d_{P\cap P'}$. Thus, \[a_*\alpha^!\Delta^!(\nu_1\times \nu_2)^!(\bar{S}\boxtimes S')[2d_1](d_1) \cong a_*\alpha^!\Delta^!(\vd\ic_\cusp\boxtimes \ic_\cusp)[-4d_P+2d_1](-2d_P+d_1)\] which is isomorphic to $ \ell_*(R\dgHom(\ic_\cusp, \ic_\cusp)).$ Finally, \eqref{eq:indtofib} yields Step 2.

Step 3: Replace cohomology with cohomology with compact support so that the coefficient is a constructible sheaf.

We make the following renormalizations of $S$ and $\bar{S}$ to ease calculations. Let $Q := S[d_p](\frac{d_p}{2})$. This is a (simple) perverse sheaf, pure of weight 0 with Verdier dual $\vd Q \cong \bar{S}[d_p](\frac{d_p}{2})$. We also let $\mathcal{F}$ denote $S[-\dim\co_L](\frac{-\dim\co_L}{2})$ and $\bar{\mathcal{F}}$ denote $\bar{S}[-\dim\co_L](\frac{-\dim\co_L}{2})$. These are both pure of weight 0 and in the constructible t-structure by cleanness of $\mathcal{C}$. Furthermore, note that $\mathbb{A}_\cusp = \mu_!\pi^* Q[d_p](\frac{d_p}{2})$. 

First, \cite[8.6.4]{CG}, $\cH_G^0(Z, \tau^! \bar{S}\boxtimes S)\cong \Hom_{\dbeq{G}{\cn}}(\mathbb{A}_\cusp, \mathbb{A}_\cusp)$, and this is isomorphic to degree 0 morphisms in the non-equivariant derived category  $\Hom_{\dbeq{c}{\cn}}(\mathbb{A}_\cusp, \mathbb{A}_\cusp)$, which is isomorphic to $\Hom_{\dbeq{c}{\cn}}(\mathbb{A}_\cusp[m](n), \mathbb{A}_\cusp[m](n))$ for any $m, n\in\mathbb{Z}$. Thus, $ \cH_G^0(Z, \tau^! \bar{S}\boxtimes S)$ is isomorphic to  $\Hom_{\dbeq{c}{\cn}}(\mu_!\pi^*Q, \mu_!\pi^*Q)$. Let $d_2 = 2d_P + \dim\co_L$. For clarity, we extend \eqref{map:tau} to the diagram below.

\begin{equation}\label{diag}
\begin{tikzcd}
Z\arrow{d}{\mu_{12}}\arrow[hookrightarrow]{r}{i} &  \widetilde{\cn}^P\times\widetilde{\cn}^P\arrow{d}{\mu\times\mu}\arrow{r}{\pi\times\pi} & \cc_{P}\times\cc_{P}\\
\cn\arrow{d}{a}\arrow[hookrightarrow]{r}{\Delta} & \cn\times \cn\\
  \pt &
\end{tikzcd} 
\end{equation} To finish step (1), we use the following chain of isomorphisms
 
\begin{align*} \Hom_{\dbeq{c}{\cn}}(\mu_!\pi^*Q, \mu_!\pi^*Q)&\\
& \hspace{-.5in}\cong\cH^0(\cn, \vd(\mu_!\pi^*Q)\otimes^!(\mu_!\pi^*Q)\\
& \hspace{-.5in}\cong \cH^0(\cn, \Delta^!(\mu\times\mu)_!(\pi\times\pi)^*( (\vd Q)\boxtimes Q)[2d_P](d_P))\\
& \hspace{-.5in}\cong \Hom(\Delta^*(\mu\times\mu)_!(\pi\times\pi)^* Q\boxtimes(\vd Q)[2d_P](4d_B +d_P), \dual{\cn})\\
& \hspace{-.5in}\cong \Hom(\mu_{12!}i^*(\pi\times\pi)^*(\mathcal{F}\boxtimes\bar{\mathcal{F}}[2d_2](d_2+ 2d_B)), \dual{\cn})\\
& \hspace{-.5in}\cong\Hom^{-2d_2}(a_!\mu_{12!}i^*(\pi\times\pi)^*(\mathcal{F}\boxtimes\bar{\mathcal{F}})(d_2+ 2d_B),\constant_{\mathrm{pt}})\\
& \hspace{-.5in}\cong\cH^{2d_2}_c(Z, i^*(\pi\times\pi)^*(\mathcal{F}\boxtimes\bar{\mathcal{F}})(d_2+ 2d_B))^* .
\end{align*}

Here, \cite[8.3.16]{CG} implies the first isomorphism. The second follows from definition of $\otimes^!$ along with the identification $\vd(\mu_!\pi^*Q)\cong \mu_!\pi^*(\vd Q)[2d_P](d_P)$ using properness of $\mu$ and smoothness of $\pi$. The third uses the isomorphism $\Hom(A, B)\cong \Hom(\vd B, \vd A)$. The fourth is by base change for diagram \eqref{diag} and renormalization. Lastly, the fifth uses $\dual{\cn}=a^!\constant_{\mathrm{pt}}$ and adjunction.

Note that $i^*(\pi\times\pi)^*(\mathcal{F}\boxtimes\bar{\mathcal{F}})(d_2+ 2d_B)$ is a constructible sheaf since the functor $ i^*(\pi\times\pi)^*(d_2+ 2d_B)$ is exact with respect to the constructible t-structure.  

Step 4: $\cH_G^0(Z, \tau^! \bar{S}\boxtimes S) \cong \bigoplus_{w}\cH_G^0(Z_{w\circ}, \iota_w^!\tau^! \bar{S}\boxtimes S).$ 
We'll prove the corresponding statement for (non-equivariant) cohomology with compact support: $\cH^{2d_2}_c(Z, i^*(\pi\times\pi)^*(\mathcal{F}\boxtimes\bar{\mathcal{F}})(d_2+ 2d_B))^*$. Let $Z^\cusp := (G\times^P(\overline{{\co}_L}+\mathfrak{u}_P)\times_\cn (G\times^P(\overline{\co_L}+\mathfrak{u}_P)$. This is a closed subvariety of $Z$ and \cite[Corollary 2.5]{DR1} implies it has dimension at most $d_2$. For $w= \mathrm{id}$, we have $Z_w\cap Z^\cusp \cong G\times^{P}(\overline{\co_L}+\mathfrak{u}_P)$ which has dimension equal to $d_2$, so we see that $Z^\cusp$ has dimension $d_2$. Note that the constructible sheaf $\mathcal{G} :=i^*(\pi\times\pi)^*(\mathcal{F}\boxtimes\bar{\mathcal{F}})(d_2+ 2d_B)$ has support contained within $Z^\cusp$, so using the open/closed long exact sequence, we see that $\cH^i_c(Z, \mathcal{G})\cong\cH^i_c(Z^\cusp, \mathcal{G})$ for all $i$. In particular, $\cH^{2d_2}_c(Z^\cusp, \mathcal{G})\cong \cH^{2d_2}_c(Z, \mathcal{G})$ decomposes according to the dimension $d_2$ irreducible components of $Z^\cusp$ since $Z^\cusp$ has dimension $d_2$. Now, it is clear that $Z^\cusp = \cup_{w} Z^\cusp\cap Z_w$. Furthermore, we know by Step 1, that our cohomology vanishes along $Z^\cusp\cap Z_w$ for $w$ bad. Recall the notation from Step 2. Using \cite[Lemma 2.2]{DR1}, for $w$ good, we have $Z^\cusp\cap Z_w = G\times^{P\cap P'}(\overline{\co_L}+\mathfrak{u}_P\cap\mathfrak{u}_{P'})$. This is irreducible so its closure is an irreducible component of $Z^\cusp$, and it has dimension $d_2$.  Finally, summing over all $w$ good, we see that the cohomology has dimension equal to $|W(L)|$. 
\end{proof}

\begin{remark} We note that the proof of the above theorem did not rely in any way on the choice of $\mathbb{F}_q$-version of the cuspidal local system. In particular, one should be able to use the techniques in Propositions \ref{propcohovanish} and \ref{prop:frobin} to show that the induction and restriction functors preserve semisimple Frobenius actions, and so satisify a version of the \textit{standard conjectures}. See \cite[Proposition 1.15]{Mi}, for instance.
\end{remark}

\begin{proposition}\label{prop:frobin} The Frobenius action on $\uHom^i(\mathbb{A}_\cusp, \mathbb{A}_\cusp(\frac{i}{2}))$ is trivial and $\mathbb{A}_\cusp$ is semisimple.
\end{proposition}
\begin{proof} First assume $i=0$. The argument in \cite[8.6.4]{CG} implies we have an isomorphism $\Hom_{\dm(\cn_\circ)}(\mathbb{A}_{\cusp\circ}, \mathbb{A}_{\cusp\circ})\cong \cH^0(Z_\circ, \tau^! \bar{S}\boxtimes S)$. By \cite[5.1.2.5]{BBD}, we have a short exact sequence relating the Frobenius coinvariants and invariants to morphisms in $\dm(\cn_\circ)$ \[0\rightarrow \uHom^{-1}(\mathbb{A}_{\cusp}, \mathbb{A}_{\cusp})_{\mathrm{Fr}}\rightarrow\Hom_{\dm(\cn_\circ)}(\mathbb{A}_{\cusp\circ}, \mathbb{A}_{\cusp\circ})\rightarrow \uHom(\mathbb{A}_{\cusp}, \mathbb{A}_{\cusp})^{\mathrm{Fr}} \rightarrow 0.\] Since $\mathbb{A}_{\cusp}$ is perverse, we see that $\uHom^{-1}(\mathbb{A}_{\cusp}, \mathbb{A}_{\cusp}) = 0$. Hence, there is an isomorphism $\Hom_{\dm(\cn_\circ)}(\mathbb{A}_{\cusp\circ}, \mathbb{A}_{\cusp\circ})\cong \uHom(\mathbb{A}_{\cusp}, \mathbb{A}_{\cusp})^{\mathrm{Fr}}$. Furthermore, we have an injection $\uHom(\mathbb{A}_{\cusp}, \mathbb{A}_{\cusp})^{\mathrm{Fr}}\hookrightarrow\uHom(\mathbb{A}_{\cusp}, \mathbb{A}_{\cusp}).$ Finally, this injection must be equality since the dimensions of $\Hom_{\dm(\cn_\circ)}(\mathbb{A}_{\cusp\circ}, \mathbb{A}_{\cusp\circ})$ and $\uHom(\mathbb{A}_{\cusp}, \mathbb{A}_{\cusp})$ are both $|W(L)|$ by Proposition \ref{propcohovanish} and \cite{L1}.

To see $\mathbb{A}_\cusp$ is semisimple, we simply note that the argument after the proof of \cite[Lemma 5.3]{Ri} applies replacing the role of $W$ with $W(L)$.

Finally, recall that we have an adjoint pair of functors $(\ind{P}{G}, \res{P}{G})$. Hence, the above argument implies $\Hom_{\dbeq{L}{\cn_{L\circ}}}(\ic_\circ, \res{P}{G}\ind{P}{G} \ic_\circ)$ also has dimension $|W(L)|$. Hence $\res{P}{G}\ind{P}{G} \ic_\circ$ is isomorphic to a direct sum of $|W(L)|$-many copies of $\ic_\circ$. Adjunction gives the first isomorphism and Lemma \ref{lem:ext} gives the second isomorphism: $\uHom^i(\mathbb{A}_{\cusp\circ}, \mathbb{A}_{\cusp\circ}) \cong \bigoplus_{W(L)}\uHom^i(\ic_{\cusp\circ}, \ic_{\cusp\circ}) \cong  \bigoplus_{W(L)} \uH^i_{Z(L)^\circ}(\mathrm{pt})$. Finally, it is well-known that Frobenius acts on $\uH^i_{Z(L)^\circ}(\mathrm{pt})$ as multiplication by $q^{i/2}$, and so Frobenius acts on  $\uHom^i(\mathbb{A}_{\cusp\circ}, \mathbb{A}_{\cusp\circ}(\frac{i}{2}))$ trivially.
\end{proof}

\section{Mixed and Derived Versions of the Generalized Springer Correspondence} Fix a cuspidal datum $\cusp = (L, \co_L, \cc)$. Our goal is to define and study a mixed version of the triangulated category $\dbeq{G}{\cn, \mathbb{A}_\cusp}$. Instead of recording full proofs here, we apply the construction and proofs in \cite[Sections 5, 6, Appendix]{Ri} to our setting with the following replacements: \begin{itemize}
\item the Springer sheaf $\mathbb{A}$ with the Lusztig sheaf $\mathbb{A}_\cusp$
\item the Weyl group $W$ with the relative Weyl group $W(L)$
\item the $G$-equivariant cohomology of the flag variety $\cH_G(G/B)\cong\Sh$ (denoted $\mathrm{S}\mathfrak{h}^*$ in \cite{Ri}) with $\cH_L(\co_L)\cong\Sz$, the $L$-equivariant cohomology of the nilpotent orbit $\co_L$  that supports the cuspidal local system $\cc$
\item the category $\dbeq{G}{G/B}\cong \dbeq{T}{\pt}$ with $\dbeq{L}{\cn_L, \ic_\cusp}\cong \dbeq{Z}{\pt}$ (see Proposition \ref{prop:cuspformal})
\item the graded ring $\R$ (denoted $\mathcal{A}_G$ in R) with $\Rc$
\item Induction and Restriction functors
\end{itemize}

As in \cite[Section 5]{Ri}, we define $\pure(\cn_\circ, \mathbb{A}_\cusp)$ as the full subcategory of $\dgm(\cn_\circ)$ where objects are finite direct sums of summands of the pure sheaves $\mathbb{A}_{\cusp\circ}[2n](n)$, $n\in\mathbb{Z}$. We denote by $\kb\pure(\cn_\circ, \mathbb{A}_\cusp)$ the corresponding homotopy category of $\pure(\cn_\circ, \mathbb{A}_\cusp)$. This triangulated category has a natural perverse t-structure and a second t-structure which we refer to as the Koszul dual t-structure. See \cite[Section 4]{Ri} and \cite[Proposition 5.4]{AR}. The idea that $\kb\pure(X_0)$ (appropriately defined) should be a good substitute for a mixed, derived version of $\dbeq{c}{X}$ is due to \cite{AR}.
\begin{theorem}\label{lem:catinv} \begin{enumerate}\item The category $\pure(\cn_\circ, \mathbb{A}_\cusp)$ is Frobenius invariant in the sense of \cite[Definition 2.3]{Ri}. \item There is a realization functor $\beta: \kb\pure(\cn_\circ, \mathbb{A}_\cusp)\rightarrow\dgm(\cn_\circ, \mathbb{A}_{\cusp\circ})$ which is perverse exact and restricts to inclusion on $\pure(\cn_\circ, \mathbb{A}_\cusp)$.\item  The triangulated category $\kb\pure(\cn_\circ, \mathbb{A}_\cusp)$ is a mixed version of $\dbeq{G}{\cn, \mathbb{A}_\cusp}$.\end{enumerate}
\end{theorem}
\begin{proof} Part (1) is implied by Theorem \ref{prop:frobin} and an argument similar to \cite[Lemma 5.4]{Ri} which computes the Frobenius action on $\Hom$ between two indecomposable objects in $\pure(\cn_\circ, \mathbb{A}_{\cusp\circ})$. The realization functor is constructed in \cite[Proposition 3.7]{Ri} assuming Frobenius invariance. Finally, \cite[Theorem 4.3]{Ri} yields part (3). 
\end{proof}

Let $\db(\Rc)$ denote the bounded derived category of finitely generated graded $\Rc$ modules. The following mixed version of Lusztig's generalized Springer correspondence holds.
\begin{theorem} \label{thm:mixgen}For each cuspidal datum $\cusp$, we have an equivalence of triangulated categories $\kb\pure(\cn_\circ, \mathbb{A}_\cusp)\cong\db(\Rc).$
\end{theorem}
\begin{proof} The proof follows by the same reasoning as \cite[Theorem 6.3, Proposition 6.5]{Ri}. The idea is to show the functor $\phi = \bigoplus_{i} \Hom(\mathbb{A}_{\cusp\circ}[2i](i), -)$ takes each $\ic_{\chi\circ}[2i](i)$ to the indecomposable projective module $\{i\}V_\chi\otimes \Sz$, where $\ic_\chi$ corresponds to the irreducible $W(L)$-representation $V_\chi$ under the generalized Springer correpondence and $\{i\}$ denotes shift by $i$. It's easy to see that $\phi$ induces an isomorphism on morphisms between objects in $\pure(\cn_\circ, \mathbb{A}_{\cusp\circ})$ and projective modules for the algebra $\Rc$.  
\end{proof}

Let $\dg(\Rc)$ denote the derived category of finitely generated dg-modules over $\Rc$, regarded as a dg-algebra with trivial differential.
Finally, we apply Section 7 and the appendix from \cite{Ri} to arrive at our main result.
\begin{theorem} \label{thm:main}The category $\dbeq{G}{\cn, \mathbb{A}_\cusp}$ is equivalent as a triangulated category to $\dg(\Rc)$.
\end{theorem}
\begin{proof} We do not give full details and instead refer the reader to \cite[Section 7, Appendix]{Ri} with the above replacements. The sketch is as follows: first take a projective resolution $\tilde{P}^\bullet$ of $\mathbb{A}_{\cusp\circ}$ in the heart of the Koszul dual t-structure. Let $P^\bullet$ denote its image under $\eta\circ\beta:\kb\pure(\cn_\circ, \mathbb{A}_{\cusp\circ})\rightarrow\dbg(\cn).$ Let $\mathcal{R}^n : = \prod_{n = i+j, k\in\mathbb{Z}}\Hom_{\dbg(\cn)}(P^{-i+k}, P^j[k])$. Then $\mathcal{R} = \bigoplus_{n\in\mathbb{Z}}\mathcal{R}^n$ is a differential graded algebra with differential given by $d_{\mathcal{R}}(f) = d_P f - (-1)^n f d_P$ for $f$ homogeneous of degree $n$. Because of Theorem \ref{lem:catinv}, $\mathcal{R}$ also has a second grading, and moreover, $\mathcal{R}$ is formal. See \cite[Theorem 7.4]{Ri} for details. Hence, \cite{BL} implies that $\dg(\mathcal{R})\cong\dg(\Rc)$. 

Now we must construct a functor from $\dbeq{G}{\cn, \mathbb{A}_\cusp}$ to $\dg(\mathcal{R})$. Let $M$ be an object in $\dbeq{G}{\cn, \mathbb{A}_\cusp}$. Define the $\mathcal{R}$ dg-module $\Phi(M)$ in degree $i$ as $\Phi(M)^i = \prod_{j\in\mathbb{Z}}\Hom_{\dbg(\cn)}(P^{-i+j}, M[j])$ with differential $d_{\Phi(M)} f = (-1)^{i+1}d_P f$ for $f$ homogeneous of degree $i$. Because our functor is defined in terms of a complex $P^\bullet$ of objects in $\dbeq{G}{\cn, \mathbb{A}_\cusp}$, it isn't obvious the functor is triangulated. It's clear that it is additive. The argument in \cite[Lemma 7.5]{Ri} implies $\Phi$ commutes with shift. We apply \cite[Appendix]{Ri} to show $\Phi$ takes distinguished triangles to distinguished triangles. Key roles are played by the adjoint pair of induction and restriction functors and the category $\dbeq{G}{G/B}\cong \dbeq{T}{\pt}$ from which the Springer sheaf $\mathbb{A}$ is induced. The proof that the functor is triangulated requires the prior knowledge that $\dbeq{T}{\pt}$ is formal and equivalent to $\dg(\Sh)$ due to \cite{BL}. In the general setting, the Lusztig sheaf $\mathbb{A}_\cusp = \ind{P}{G}\ic_\cusp$ is induced from the category $\dbeq{L}{\cn_L, \ic_\cusp}$. By Proposition \ref{prop:cuspformal}, we have $\dbeq{L}{\cn_L, \ic_\cusp}\cong \dbeq{Z}{\pt}\cong\dg(\Sz).$ Hence, the proof in the appendix applies in our setting as well. 

Finally, since we know the functor is triangulated, Beilinson's Lemma implies that it suffices to check morphism matching on a collection of objects that generate the categories. This is outlined in \cite[Theorem 7.9]{Ri}.
\end{proof}

\section*{Acknowledgements} We are very grateful to George Lusztig for answering questions about his work; in particular, for pointing us to \cite[Proposition 4.7]{L3}. We'd also like to thank Matt Douglass and Pramod Achar for helpful discussions.



\end{document}